\documentclass[11pt,reqno]{amsart}
\usepackage{amscd,amsmath,amssymb,amsthm,amsfonts,epsfig,graphics}
\usepackage{geometry}
\usepackage{graphicx}
\usepackage{mathrsfs,amssymb}
\usepackage{bm}
\usepackage{url}
\usepackage{subfigure}
\usepackage{hyperref}
\usepackage{color}

\newtheorem{theorem}{Theorem}[section]
\newtheorem{conjecture}[theorem]{Conjecture}
\newtheorem{corollary}[theorem]{Corollary}
\newtheorem{definition}[theorem]{Definition}
\newtheorem{lemma}[theorem]{Lemma}
\newtheorem{proposition}[theorem]{Proposition}
\newtheorem{remark}[theorem]{Remark}
\numberwithin{equation}{section}

\newcommand{\ls}{\lesssim}
\newcommand{\gs}{\gtrsim}
\newcommand{\R}{\mathbb R}
\newcommand{\Z}{\mathbb Z}
\newcommand{\N}{\mathbb N}
\newcommand{\T}{\mathbb T}

\newcommand{\eps}{\varepsilon}
\newcommand{\la}{\lambda}

\newcommand{\dd}{\,d}

\newcommand{\E}{\mathcal{E}}
\DeclareMathOperator{\spanop}{span}

\DeclareMathOperator{\vol}{vol}
\DeclareMathOperator{\dist}{dist}
\allowdisplaybreaks[4]
\begin{document}
	\title[Restriction of toral eigenfunctions]
	{Restriction estimates for toral eigenfunctions and lattice points in  spherical regions}
	


	\author{ Cheng Zhang and Zhifei Zhu}
	
	\address{Mathematical Sciences Center\\
		Tsinghua University\\
		Beijing, BJ 100084, China}
	\email{czhang98@tsinghua.edu.cn; zhifeizhu@tsinghua.edu.cn}

	

	\dedicatory{}

	\begin{abstract}
		We establish new $L^2$ restriction estimates for toral eigenfunctions. These estimates are sharp in certain cases, and thus prove a conjecture of Huang-Zhang for smooth submanifolds of large codimension. In particular, they provide new progress toward a conjecture of Bourgain-Rudnick. The proof combines a slicing and packing method with the approximation of the discrete spherical multiplier by Magyar-Stein-Wainger and Magyar.
	\end{abstract}
	
	\maketitle
	
	\section{Introduction}
	
	Let \(M\) be a compact smooth Riemannian manifold without boundary, and let
	\(\Delta\) be its Laplace--Beltrami operator. Suppose
	\[
	-\Delta e_\lambda=\lambda^2 e_\lambda,\qquad \lambda>1,
	\]
	one may ask how large the restriction of \(e_\lambda\) can be on a fixed
	submanifold \(\Sigma\subset M\).  Let \(d=\dim M\), let \(k=\dim\Sigma\), and
	write \(m=d-k\) for the codimension.  Burq--G\'erard--Tzvetkov
	\cite{BGT2006} and Hu \cite{Hu2009} proved that
	\begin{equation}\label{BGT01}
		\|e_\lambda\|_{L^2(\Sigma)}
		\lesssim
		\Lambda(m,\lambda)\|e_\lambda\|_{L^2(M)},
	\end{equation}
	where
	\[
	\Lambda(m,\lambda)=
	\begin{cases}
		\lambda^{1/4}, & m=1,\\
		\lambda^{1/2}\sqrt{\log\lambda}, & m=2,\\
		\lambda^{(m-1)/2}, & m\ge3.
	\end{cases}
	\]
	These estimates are sharp on the standard sphere, apart from the logarithmic
	loss in codimension two.  In the case of codimension one, if \(\Sigma\) has positive or
	negative definite second fundamental form, the exponent \(1/4\) can be
	improved to \(1/6\).  
	
	Related developments include the microlocal background in the work of Greenleaf--Seeger \cite{gs1994} and Tataru \cite{tata}. Quasimode and
	curved-hypersurface variants were developed by Tacy \cite{tacy} and
	Hassell--Tacy \cite{HT12}. For curves and geodesics, see Reznikov
	\cite{Rez2004}, Bourgain \cite{bo2009}, Sogge--Zelditch \cite{sz2014},
	Chen \cite{chen}, Chen--Sogge \cite{chensogge}, Xi--Zhang \cite{xz},
	Blair \cite{mb}, Zhang \cite{zhang}, and Park \cite{park}.  In the flat
	torus setting, one may refer to Bourgain--Rudnick \cite{BR2011}, Hezari--Rivi\`ere \cite{hr},
	Huang--Zhang \cite{hzapde}, Wang--Zhang \cite{wzadv}, and
	Burq--Germain--Sorella--Zhu \cite{BGSZ}. Other variants include \cite{hez,hr,BP,hwz,yzhang}.
	
	Let $M=\mathbb{T}^d=\mathbb{R}^d/\mathbb{Z}^d$ be the standard flat torus. We write $\lambda$ for the frequency radius, so the Laplace eigenvalue is $(2\pi\lambda)^2$. In this setting, eigenfunctions are
	trigonometric polynomials whose frequencies lie on a lattice sphere:
	\[
	e_\lambda(x)=\sum_{\xi\in \mathcal E_\lambda}a_\xi e^{2\pi i\xi\cdot x},
	\qquad
	\mathcal E_\lambda=\mathbb Z^d\cap \lambda S^{d-1}.
	\]
	Thus, restriction estimates of eigenfunctions are closely connected with the distribution of the lattice points in $\la S^{d-1}$.
	This arithmetic structure allows one to improve general estimates
	\eqref{BGT01}.
	
	For curves
	\(\Sigma\subset\mathbb T^2\), Burq--G\'erard--Tzvetkov observed that
	\eqref{BGT01} can be improved to
	\begin{equation}\label{BGTtorus}
		\|e_\lambda\|_{L^2(\Sigma)}
		\le C_\eps\lambda^\eps
		\|e_\lambda\|_{L^2(\mathbb T^2)},
		\qquad \forall \eps>0,
	\end{equation}
	using the corresponding \(L^\infty\) bound for eigenfunctions. They asked whether the factor \(\lambda^\eps\) can be removed. Bourgain--Rudnick
	\cite{BR2011} proved such a uniform bound for curve segments with nonvanishing
	geodesic curvature, and Huang--Zhang \cite[Theorem 2]{hzapde} proved it for
	the closed geodesics case. For a geodesic segment on \(\mathbb T^2\), this problem is closely related to the open problem of whether every
	arc of length \(\lambda^{1/2}\) on the circle \(|x|=\lambda\) contains only
	\(O(1)\) lattice points.
	
	Bourgain--Rudnick proposed the following higher-dimensional conjecture for real-analytic hypersurfaces.
	
	\begin{conjecture}[Bourgain--Rudnick \cite{BR2011}]\label{BRconj}
		Let \(d\ge2\), and let \(\Sigma\subset\mathbb T^d\) be a real-analytic
		hypersurface.  Then for $\lambda>1, $
		\begin{equation}\label{BR01}
			\|e_\lambda\|_{L^2(\Sigma)}
			\lesssim
			\|e_\lambda\|_{L^2(\mathbb T^d)}.
		\end{equation}
		If moreover \(\Sigma\) has nowhere vanishing curvature and
		\(\lambda>\lambda_\Sigma\), then
		\begin{equation}\label{BR02}
			\|e_\lambda\|_{L^2(\Sigma)}
			\gtrsim
			\|e_\lambda\|_{L^2(\mathbb T^d)}.
		\end{equation}
	\end{conjecture}
	
	Bourgain--Rudnick proved the conjecture for \(d=2,3\) in the real-analytic
	case with nowhere vanishing curvature.  Hezari--Rivi\`ere \cite{hr} obtained
	the upper bound along a density-one subsequence of eigenfunctions for smooth
	embedded hypersurfaces with nonvanishing principal curvatures.
	
	Recently, Huang--Zhang \cite{hzapde} proposed the following analogue in the case of smooth manifolds in higher
	codimensions.
	
	\begin{conjecture}[Huang--Zhang \cite{hzapde}]\label{HZconj}
		Let \(d\ge3\), and let \(\Sigma\subset\mathbb T^d\) be a smooth
		submanifold of codimension \(m\ge2\).  Then, for \(\lambda>1 \) and \(\eps>0\),
		\begin{equation}\label{hzeq}
			\|e_\lambda\|_{L^2(\Sigma)}
			\lesssim_\eps
			\lambda^{\frac m2-1+\eps}
			\|e_\lambda\|_{L^2(\mathbb T^d)}.
		\end{equation}
	\end{conjecture}
	
	Huang--Zhang proved this
	bound \eqref{hzeq} as well as \eqref{BR01} for rational totally geodesic submanifolds, while the problem remains
	open for general submanifolds.  The exponent is also sharp in general; see
	Proposition \ref{prop1} below.
	
	Our first main result is Theorem \ref{thm1}, which proves Conjecture \ref{HZconj} for $m\ge \frac{d+3}2$ when $d\ge5$. It also improves the restriction estimates \eqref{BGT01} for totally geodesic submanifolds of all codimensions. In addition, we prove new restriction estimates for curves on $\mathbb{T}^3$; see Theorem \ref{T3thm}.

	\begin{theorem}\label{thm1}
		Let \(d\ge3\).
		
		\emph{(i)} Let \(\Sigma\subset\mathbb T^d\) be a compact totally geodesic
		submanifold of codimension \(m\).  Then for every \(\eps>0\),
		\[
		\|e_\lambda\|_{L^2(\Sigma)}
		\lesssim_\eps
		\lambda^{\alpha(m,d)+\eps}
		\|e_\lambda\|_{L^2(\mathbb T^d)},
		\]
		where
		\[
		\alpha(1,d)=
		\begin{cases}
			\frac1{12},& d=3,\\
			\frac18,& d\ge4,
		\end{cases}
		\qquad
		\alpha(2,d)=
		\begin{cases}
			\frac13,& d=3,\\
			\frac38,& d\ge4,
		\end{cases}
		\]
		\[
		\alpha(3,d)=
		\begin{cases}
			\frac34,& d=4,\\
			\frac45,& d=5,\\
			\frac{13}{16},& d\ge6,
		\end{cases}
		\]
		and, for \(m\ge4\) and \(d\ge5\),
		\[
		\alpha(m,d)=
		\begin{cases}
			\frac m2-1,& d\le 2m-3,\\
			\frac m2-\frac34,& d=2m-2,\\
			\frac m2-\frac34+2^{-m-1},& d\ge 2m-1.
		\end{cases}
		\]
		
		\emph{(ii)} If \(d\ge5\), \(m\ge \frac{d+3}2\), and
		\(\Sigma\subset\mathbb T^d\) is any compact smooth submanifold of
		codimension \(m\), then for every \(\eps>0\)
		\[
		\|e_\lambda\|_{L^2(\Sigma)}
		\lesssim_\eps
		\lambda^{\frac m2-1+\eps}
		\|e_\lambda\|_{L^2(\mathbb T^d)}.
		\]
	\end{theorem}
	
	Part (ii) proves Conjecture \ref{HZconj} in the range
	\(m\ge(d+3)/2\), \(d\ge5\).  Part (i) gives new bounds for totally geodesic
	submanifolds in all codimensions.  When \(d=3\), they agree with earlier bounds of Huang--Zhang
	\cite[Theorem 4]{hzapde}. They also improve the general
	Burq--G\'erard--Tzvetkov estimate \eqref{BGT01}.

	The proof of Theorem \ref{thm1} combines a slicing and packing method with the approximation of the discrete spherical multiplier (Proposition \ref{propm}) due to Magyar-Stein-Wainger \cite{msw} and Magyar \cite{m}. At present, our approach only partially confirms Conjecture \ref{HZconj}. For smaller codimension, the remainder term \(\lambda^{(d-1)/2+\eps}\) in Proposition \ref{propm} prevents
	this method from reaching the conjectural exponent. A full resolution of this conjecture would require a more efficient use of cancellation in the relevant exponential sums. 
	
	We also prove, in Theorem \ref{T3thm}, new restriction estimates for curves in \(\mathbb T^3\).  If
	\(\Sigma\subset\mathbb T^3\) is a geodesic segment, a smooth curve segment
	with nowhere vanishing torsion, or a real-analytic curve segment whose geodesic curvature is nowhere vanishing, then
	\begin{equation}\label{eqt3}
		\|e_\lambda\|_{L^2(\Sigma)}
		\lesssim_\eps
		\lambda^{1/3+\eps}
		\|e_\lambda\|_{L^2(\mathbb T^3)}.
	\end{equation}
	If \(\Sigma\) is a smooth planar curve segment with nowhere vanishing geodesic
	curvature, then we can obtain the stronger bound $O_\eps(\la^{1/4+\eps})$. 
	These estimates improve the general codimension-two bound in \eqref{BGT01}.
	Our proof uses Bourgain--Rudnick's lattice-point estimates in spherical caps
	together with the van der Corput lemma. Notably, the bound \eqref{eqt3} agrees with the one predicted by the endpoint Discrete Restriction Conjecture (Bourgain--Demeter
	\cite[Conjecture~2.6]{BD}). See Remark~\ref{dcrem}.
	
	\subsection{Paper structure}
	
	Section \ref{sec2} introduces the lattice-point counting problem associated with
	totally geodesic restrictions and proves the reduction from restriction
	estimates to band-counting estimates.  It also records lower bounds showing
	that the expected exponents are sharp.  Section \ref{sec3} proves the geometric
	lattice-point estimates by slicing and packing.  Section \ref{sec4} uses the
	Magyar--Stein--Wainger and Magyar spherical multiplier approximation to prove
	restriction estimates for general smooth submanifolds.  Section \ref{sec5} combines the
	geometric and multiplier estimates to prove Theorem \ref{thm1}.  Section \ref{sec6}
	proves the curve estimates on \(\mathbb T^3\) and discusses their relation to
	the Discrete Restriction Conjecture.
	
	\subsection{Notation}
	
	Throughout the paper, \(X\lesssim Y\) means that \(X\le CY\), where the
	constant \(C\) is independent of \(\lambda\). The implicit constant \(C\) may depend on the dimensional constant \(d\), the fixed band-width
	constant \(C_0\), and the transversality constant \(\nu\) whenever they appear. 
	
	The subscripts indicate the allowed
	dependence of the implicit constant. Thus, \(X\lesssim_\eps Y\) means that the
	constant may depend on \(\eps\).  We write \(X\approx Y\) if
	\(X\lesssim Y\) and \(Y\lesssim X\).  Finally, \(X\gg Y\) means that
	\(X\ge CY\) for a sufficiently large constant \(C>0\).
	
	\subsection{Acknowledgments}
	
	The authors are supported in part by the National Key R\&D Program of China
	2024YFA1015300.  C.Z. is also supported in part by NSFC Grant 12371097.
	Z.Z. is also supported in part by NSFC Grant 12501065. The authors would like to thank Xiaocheng Li for his helpful comments on the manuscript.
	
	\section{Preliminaries}\label{sec2}
	Let $\lambda S^{d-1}$ be the standard sphere in the Euclidean space $\R^d$ of radius $\lambda$. If \(\la^2\notin\N\), then
	\(\Z^d\cap\la S^{d-1}=\varnothing\), so all estimates are trivial.  We
	therefore assume \(\la^2\in\N\).
	
	A unit band on \(\la S^{d-1}\) is a set of the form
	\[
	B(u,x_0)=\left\{x\in\la S^{d-1}:
	\left|\frac{u}{|u|}\cdot(x-x_0)\right|\le C_0\right\}.
	\]
	For \(1\le k\le d-1\), the bands are \(\nu\)-transverse if their unit
	normals \(n_j=u_j/|u_j|\) satisfy
	\[
	|n_1\wedge\cdots\wedge n_k|\ge\nu.
	\]
	Define
	\[
	A_{k,d,\la}:=\sup\#\left(\Z^d\cap\la S^{d-1}\cap
	\bigcap_{j=1}^k B(u_j,x_{0,j})\right),
	\]
	where the supremum is over all \(\nu\)-transverse \(k\)-tuples of bands. Indeed, $A_{k,d,\la}$ is essentially the maximum number of lattice points in the intersection of the sphere $\la S^{d-1}$ and the 1-neighborhood of a $(d-k)$-plane in $\mathbb{R}^d$. See Lemma \ref{lem:transverse-slab}.

	\subsection{Lower bounds}
	For $n\in\N$, write
	\[
	r_d(n)=\#\{x\in\Z^d: |x|^2=n\}.
	\]
	\begin{lemma}\label{lem:rd-odd-lower}
		Let $d\ge4$ and let $\la \ge3$ be odd.  Then $r_d(\la^2)\gs \la^{d-2}.$ 
	\end{lemma}
	This lemma is known and can be proved using Jacobi's four-square formula. For $d=4$ and even $\la$, the lower bound does not need to hold. See Walfisz \cite{w} and Bourgain-Rudnick \cite[Section 2.1]{BR2011}.
	\begin{lemma}\label{lem:slice-pigeonhole}
		Let $d\ge4$, $1\le k\le d-2$, and let $\la\ge3$ be odd.  Then there exists $a_0\in\Z^k$ such that
		\[
		\#\{b\in\Z^{d-k}: |a_0|^2+|b|^2=\la^2\}
		\gs \la^{d-k-2}.
		\]
	\end{lemma}
	
	\begin{proof}
		Write $x=(a,b)$ with $a\in\Z^k$ and $b\in\Z^{d-k}$.  For each $a\in\Z^k$, set
		\[
		F(a)=\#\{b\in\Z^{d-k}: |a|^2+|b|^2=\la^2\}.
		\]
		Then
		\[
		r_d(\la^2)=\sum_{\substack{a\in\Z^k\\ |a|\le \la}}F(a).
		\]
		The number of possible $a$ is $\#\{a\in\Z^k: |a|\le \la\}
		\ls \la^k.$
		By Lemma \ref{lem:rd-odd-lower}, $r_d(\la^2)\gs \la^{d-2}.$
		Therefore, the pigeonhole principle gives some $a_0\in\Z^k$ with $F(a_0)
		\gs
		\la^{d-k-2}.$
	\end{proof}
	
	\begin{proposition}\label{thm:limsup}
		Let $d\ge3$, let $1\le k\le d-2$, and fix $C_0>0$ and $0<\nu\le1$.  Then
		\[
		\limsup_{\lambda\to\infty}
		\frac{A_{k,d,\lambda}}{\lambda^{d-k-2}}
		\gs1.
		\]
		In fact, for $d\ge4$ the lower bound holds for all sufficiently large odd integers $\lambda$.
	\end{proposition}
	
	\begin{proof}
		We first treat $d\ge4$.  Let $\la\ge3$ be odd.  By Lemma \ref{lem:slice-pigeonhole}, there exists $a_0=(a_{0,1},\dots,a_{0,k})\in\Z^k$ such that
		\[
		\#\{b\in\Z^{d-k}: |a_0|^2+|b|^2=\la^2\}
		\gs \la^{d-k-2}.
		\]
		For $j=1,\dots,k$, define the coordinate band
		\[
		B_j
		=
		\{x\in \la S^{d-1}: |x_j-a_{0,j}|\le C_0\}.
		\]
		This is a unit band with normal $e_j$.  The normals satisfy $|e_1\wedge\cdots\wedge e_k|=1,$
		so the bands are $\nu$-transverse for every $0<\nu\le1$.
		
		Every lattice point of the form $x=(a_0,b)$ with $|a_0|^2+|b|^2=\la^2$
		lies in $B_1\cap\cdots\cap B_k$. Hence
		\[
		A_{k,d,\la}
		\ge
		\#\{b\in\Z^{d-k}: |a_0|^2+|b|^2=\la^2\}
		\gs \la^{d-k-2}.
		\]
		It remains to treat $d=3$ and $k=1$.
		For every integer $\la\ge1$, the point $(\la,0,0)\in\Z^3\cap \la S^2.$
		The coordinate band $B=\bigl\{x\in \la S^2: |x_1-\la|\le C_0\bigr\}$
		contains this point.  Thus $A_{1,3,\la}\ge1.$
	\end{proof}
	
	Therefore, the following conjecture on the upper bounds seems reasonable.
	\begin{conjecture}\label{Aconj}
		For $d\ge2$ and $1\le k\le d-1$,
		\begin{equation}
			A_{k,d,\la}\ls1,\ \ \text{when}\ d-k=1,
		\end{equation}
		and
		\begin{equation}
			A_{k,d,\la}\ls_{\eps}\la^{d-k-2+\eps},\ \forall \eps>0,\ \ \text{when}\ d-k\ge2.
		\end{equation}
	\end{conjecture}
	
	\subsection{Sharp restriction estimates}
	Throughout, the totally geodesic submanifolds are assumed to be bounded with fixed unit diameter, and they are essentially determined by their directions. Huang--Zhang \cite[Theorem 3]{hzapde} reduced the restriction problem to the number $A_{k,d,\la}$, with a logarithmic loss. We can now remove this loss. The idea is to majorize the cutoff function by a function with compact Fourier support.
	
	\begin{theorem}\label{prop3a}\label{prop:logfree-reduction}
		Let $d\ge3$, and let $\Sigma\subset \mathbb{T}^d$ be a totally geodesic submanifold of dimension $k$ and fixed unit diameter. Then
		\begin{equation}\label{Tdk}
			\|e_\lambda\|_{L^2(\Sigma)}\ls
			\sqrt{A_{k,d,\lambda}}
			\|e_\lambda\|_{L^2({\mathbb{T}^d})}.
		\end{equation}
		Here the constant is independent of $\lambda$ and the direction of $\Sigma$.
		Moreover, for any fixed eigenvalue $\lambda$, there exist an eigenfunction $e_\lambda$ and a totally geodesic submanifold $\Sigma$ of dimension $k$ and fixed unit diameter  such that
		\begin{equation}\label{Tdksharp}
			\|e_\lambda\|_{L^2(\Sigma)}\gs
			\sqrt{A_{k,d,\lambda}}\|e_\lambda\|_{L^2({\mathbb{T}^d})}.
		\end{equation}
		Thus, the upper bound \eqref{Tdk} is sharp.
	\end{theorem}
	Therfore, Conjecture \ref{Aconj} implies the upper bounds in Conjectures \ref{BRconj} and \ref{HZconj} for totally geodesic submanifolds.
	\begin{proof}
		We prove the upper bound. After covering $\Sigma$ by $O(1)$ coordinate patches, it is enough to treat a patch of the form
		\[
		x=x_0+Lt,
		\qquad t\in\Omega\subset\mathbb R^k,
		\]
		where $L:\mathbb R^k\to\mathbb R^d$ is an isometric embedding and $\Omega$ has bounded diameter. Thus, the columns $Le_1,\ldots,Le_k$ are orthonormal tangent vectors. Write
		\[
		e_\lambda(x)=\sum_{n\in \E_\lambda} a_n e^{2\pi i n\cdot x},
		\qquad \E_\lambda=\{n\in\mathbb Z^d: |n|=\lambda\}.
		\]
		Then
		\[
		\int_\Sigma |e_\lambda|^2\,d\sigma
		\ls
		\int_\Omega
		\Big|\sum_{n\in \E_\lambda}a_n e^{2\pi i n\cdot x_0}e^{2\pi i t\cdot L^Tn}\Big|^2dt.
		\]
		
		Choose a fixed nonnegative function $\Psi\in L^1(\mathbb R^k)$ such that $\Psi\ge \mathbf1_\Omega$, $\widehat\Psi\ge0$, and $\widehat\Psi$ is supported in a fixed small cube $[-c_0,c_0]^k$, with $c_0\le C_0$. Expanding the square after majorizing $\mathbf 1_\Omega$ by $\Psi$ gives
		\[
		\int_\Omega
		\Big|\sum_{n\in \E_\lambda}a_n e^{2\pi i n\cdot x_0}e^{2\pi i t\cdot L^Tn}\Big|^2dt
		\le
		\sum_{m,n\in \E_\lambda} a_m\overline{a_n}e^{2\pi i(m-n)\cdot x_0}\widehat\Psi(L^T(n-m)).
		\]
		Here $\widehat\Psi(L^T(n-m))$ is supported where
		\[
		|(n-m)\cdot Le_j|\le c_0,
		\qquad j=1,\ldots,k.
		\]
		For each fixed $m$, the possible $n$ therefore lies in the intersection of the $k$ unit-width bands on $\lambda S^{d-1}$ centered at $m$ with normals $Le_1,\ldots,Le_k$. Since these normals are orthonormal, they are transverse, and the number of such $n$ is at most $A_{k,d,\lambda}$. We have
		\[\max_{m\in E_\la}\sum_{n\in E_\la}|\widehat\Psi(L^T(n-m))|\ls A_{k,d,\la}.\]
		Schur's test gives
		\[
		\sum_{m,n\in \E_\lambda}\left|a_m\overline{a_n}\widehat\Psi(L^T(n-m))\right|
		\ls A_{k,d,\lambda}\sum_{n\in \E_\lambda}|a_n|^2.
		\]
		This proves \eqref{Tdk}.
		
		Next, we prove the lower bound. Choose a transverse \(k\)-tuple of bands and a set
		\[
		P\subset \E_\lambda\cap\bigcap_{j=1}^k B(u_j,x_{0,j}),\ \text{with}\  \#P=A_{k,d,\la}.
		\]
		Let \(n_j=u_j/|u_j|\), and set
		\(V=\operatorname{span}(n_1,\ldots,n_k)\).  Since
		\(|n_1\wedge\cdots\wedge n_k|\ge\nu\), the map
		\[
		v\in V\longmapsto (n_1\cdot v,\ldots,n_k\cdot v)\in\mathbb R^k
		\]
		has bounded inverse norm. Hence, for all
		\(p,p'\in P\), $|\pi_V(p-p')|\ls1.$
		Choose an isometric embedding \(L:\mathbb R^k\to\mathbb R^d\) with
		range \(V\).  Let \(\Omega_0\subset\mathbb R^k\) be a fixed ball of sufficiently
		small radius, depending only on \(d,k,\nu,C_0\), such that
		\[
		|t\cdot L^T(p-p')|\ll 1
		\qquad
		(t\in\Omega_0,\ p,p'\in P).
		\]
		Let \(\Sigma\) be a fixed unit-size totally geodesic patch parallel to \(V\)
		and containing the subpatch \(L\Omega_0\).  Define
		\[
		e_\lambda(x)=A_{k,d,\la}^{-1/2}\sum_{p\in P} e^{2\pi ip\cdot x}.
		\]
		Then $\|e_\la\|_{L^2(\mathbb{T}^d)}= 1$.
		On the subpatch \(x=Lt\), \(t\in\Omega_0\), we have
		\[
		\begin{aligned}
			\|e_\lambda\|_{L^2(\Sigma)}^2\ge	\int_{\Omega_0}|e_\lambda(Lt)|^2\,dt
			&=
			A_{k,d,\la}^{-1}\sum_{p,p'\in P}
			\int_{\Omega_0} e^{2\pi it\cdot L^T(p-p')}\,dt \gtrsim A_{k,d,\la}
		\end{aligned}
		\]
		by the choice of \(\Omega_0\).  
		This proves \eqref{Tdksharp}.
	\end{proof}

	\begin{proposition}\label{prop1}
		Let \(\Sigma\subset\mathbb T^d\) be a fixed totally geodesic submanifold of dimension \(k\), with \(m=d-k\ge1\).  Then there are sequence \(\lambda_j\to\infty\) and $L^2$-normalized eigenfunctions \(e_{\lambda_j}\) such that
		\begin{equation}\label{eq:every-sigma-lower}
			\|e_{\lambda_j}\|_{L^2(\Sigma)}
			\gs
			1+\lambda_j^{\frac m2-1}.
		\end{equation} 
	\end{proposition}
	\begin{proof}
		If \(m\le2\),  we choose
		\(e_{\lambda_j}(x)=e^{2\pi ij n_0\cdot x}\), with \(0\neq n_0\in\mathbb Z^d\) fixed
		and \(\lambda_j=2\pi j|n_0|\), and then $\|e_{\lambda_j}\|_{L^2(\Sigma)}\gs 1.$
		It remains to consider \(m\ge3\) and $d\ge4$.

		Let $U$ be the tangent $k$-plane of $\Sigma$.  Since \(\Sigma\) is totally geodesic, on a coordinate patch we may write
		\[
		x=x_0+Lt,
		\qquad
		t\in\Omega\subset\R^k,
		\]
		where \(L:\R^k\to\R^d\) is an isometric embedding with range \(U\). We may choose a ball $B(t_0,r_0)\subset\Omega$
		with \(r_0>0\).  Choose small \(\delta>0\) such that $\delta r_0\ll1$.  
		
		For \(d\ge4\) and odd \(\lambda\), Lemma \ref{lem:rd-odd-lower} gives \(\#\E_\lambda=r_d(\lambda^2)\gs \lambda^{d-2}\). Since the projection of \(\E_\lambda\) onto $U$ lies in a \(k\)-dimensional ball of radius \(\lambda\), it may be covered by \(O(\lambda^k)\) $\delta$-balls in \(U\). By the pigeonhole principle, one $\delta$-ball contains the projections of \(\gs \la^{d-k-2}\) points of $\E_\la$. We denote this set of points of $\E_\la$ by $S_\la$, and set $N_\la=\# S_\la\gs \la^{d-k-2}$. 
		Then there is a vector \(\eta_\lambda\in\R^k\) such that
		\[
		|L^Tn-\eta_\lambda|\le \delta,
		\qquad n\in S_\lambda.
		\]
		
		Let
		\[
		e_{\lambda}(x)=
		N_\lambda^{-1/2}
		\sum_{n\in S_\lambda}
		e^{-2\pi in\cdot (x_0+Lt_0)}e^{2\pi in\cdot x}.
		\]
		Then $\|e_{\lambda}\|_{L^2(\T^d)}^2
		=1$.
		We have
		\[
		|e_{\lambda}(x_0+Lt)|=N_\lambda^{-1/2}\Big|\sum_{n\in S_\la}e^{2\pi i(t-t_0)\cdot(L^Tn-\eta_\la)}\Big|\gs N_\lambda^{1/2}\,
		\qquad t\in B(t_0,r_0).
		\]
		Therefore,
		\begin{align*}
			\|e_{\lambda}\|_{L^2(\Sigma)}^2\ge
			\int_{B(t_0,r_0)} |e_{\lambda}(x_0+Lt)|^2\dd t \gs N_\la\gs \lambda^{d-k-2}.
		\end{align*}
		This yields \eqref{eq:every-sigma-lower}.
	\end{proof}
	
	\section{Slicing and packing}\label{sec3}
	
	In this section, we use a slicing and packing method to establish new estimates for the number $A_{k,d,\la}$. Throughout this section, if $U\subset\R^d$ is a linear subspace, we write
	$\pi_U:\R^d\to U$ for the orthogonal projection onto $U$. We first show some elementary estimates for the number of integer points in $\lambda S^{d-1}$. 
	
	We shall use affine 2-planes to slice the spherical regions. Note that any affine 2-plane intersects the sphere $\la S^d$ in a circle of radius at most $\la$. The following Lemma~\ref{lem:two-parameter-sphere} follows immediately from \cite[Lemma~4]{hzapde}, we include the statement for the sake of completeness.
	
	\begin{lemma}\label{lem:two-parameter-sphere}
		If $\Gamma\subset \mathbb{R}^d$ is an affine 2-plane, then for every $\eps>0$,
		\begin{equation}\label{eq:rank-two-section}
			\#\bigl(\Gamma\cap\Z^d\cap\lambda S^{d-1}\bigr)\ls_\eps \lambda^\eps.
		\end{equation}
	\end{lemma}
	The next lemma is a basic volume estimate for counting lattice points.
	\begin{lemma}\label{lem:rotated-box}
		Let \(\R^d=U\oplus U^\perp\) be an orthogonal decomposition with \(\dim U=q\).  If
		\(L\ge1\), \(C\ge1\), and \(z\in\R^d\), then
		\[
		\#\{m\in\Z^d:|\pi_U(m-z)|\le L,
		|\pi_{U^\perp}(m-z)|\le C\}
		\ls L^q C^{d-q}.
		\]
		The implicit constant only depends on $d$ and $q$.
	\end{lemma}
	
	\begin{proof}
		Let
		\[
		K=\{x:|\pi_Ux|\le L,\ |\pi_{U^\perp}x|\le C\}.
		\]
		Since \(L,C\ge1\), the convex symmetric body \(K\) contains the Euclidean unit
		ball.  The unit cubes centered at the lattice points in \(z+K\) are disjoint and
		are contained in \(z+K+B(0,\sqrt d/2)\subset z+(1+\sqrt d/2)K\). Therefore the
		number of such lattice points is bounded by a dimensional constant times the volume
		\(\vol(K)\approx L^qC^{d-q}\).
	\end{proof}
	\begin{definition}\label{def:T}{\rm 
			For \(0\le q\le d-1\), \(1\le L\le\la^{1/2}\), and fixed \(C\ge1\), let
			\(T_{d,q}(\la,L;C)\) be the supremum of \(\#P\) over all finite sets
			\(P\subset\Z^d\cap\la S^{d-1}\) for which there is an orthogonal decomposition
			\(\R^d=U\oplus U^\perp\), \(\dim U=q\), such that
			\[
			\operatorname{diam}(\pi_UP)\le L,
			\qquad
			\operatorname{diam}(\pi_{U^\perp}P)\le C.
			\]
			We call such sets $P$  \textbf{admissible} for \(T_{d,q}(\la,L;C)\).}
	\end{definition}
	These admissible sets $P$ will serve as standard bricks to count lattice points in spherical regions. We first show an estimate for the upper bound of \(T_{d,q}(\la,L;C)\). 
	
	\begin{proposition}\label{prop:fibration}
		For every \(\eps>0\), we have
		\[
		T_{d,q}(\la,L;C)\ls_\eps L^{q(d-2)/d}\la^\eps .
		\]
	\end{proposition}
	
	\begin{proof}
		Let $V=U^\perp$. For \(a\in\Z^d\), set
		\[
		\rho(a)=C|\pi_Va|+L|\pi_Ua|.
		\]
		And consider the map \(T:\R^d\to\R^d\) defined by \(T(a)=C\pi_V(a)+L\pi_U(a).\)
		Since $V$ are orthogonal to $U$, we have
		\[
		|T(a)|\le \rho(a)\le \sqrt2 |T(a)| \quad \text{ and } \quad |\det T|=C^{d-q}L^q.
		\]
		We can choose linearly independent \(a_1,\ldots,a_d\in\Z^d\) so that
		\(\rho(a_1)\le\cdots\le\rho(a_d)\), with
		\[
		\prod_{j=1}^d \rho(a_j)\ls L^q.
		\]
		By definition, \(C,L\ge1\), and hence
		\(
		\rho(a_j)\ge |a_j|\ge1
		\)
		for every \(j\). Therefore
		\[
		\prod_{j=1}^{d-2}\rho(a_j)
		\le
		\Big(\prod_{j=1}^{d}\rho(a_j)\Big)^{(d-2)/d}
		\ls L^{q(d-2)/d}.
		\]
		For \(x,y\in P\), the assumptions give
		\[
		|\pi_V(x-y)|\le C,
		\qquad
		|\pi_U(x-y)|\le L.
		\]
		Hence, for \(1\le j\le d-2\),
		\[
		|a_j\cdot(x-y)|
		\le
		|\pi_Va_j|\,|\pi_V(x-y)|
		+
		|\pi_Ua_j|\,|\pi_U(x-y)|
		\le \rho(a_j).
		\]
		Since \(a_j\cdot x\) is an integer, it takes at most \(O(\rho(a_j))\)
		different values as \(x\) runs over \(P\). Thus the number of possible value
		lists
		\(
		(a_1\cdot x,\ldots,a_{d-2}\cdot x)
		\)
		is at most
		\[
		\prod_{j=1}^{d-2}O(\rho(a_j))
		\ls L^{q(d-2)/d}.
		\]
		Fix one of such value lists.  The remaining points satisfy
		\[
		a_j\cdot x=m_j,
		\qquad 1\le j\le d-2,
		\]
		for fixed integers \(m_j\).  The corresponding linear
		equations have two free integer variables, and thus determine an affine 2-plane. Consequently by Lemma~\ref{lem:two-parameter-sphere}, this fixed value list contributes at
		most \(O_\eps(\la^\eps)\) points on \(\la S^{d-1}\).
		Summing over all value lists gives
		\[
		T_{d,q}(\la,L;C)
		\ls_\eps
		L^{q(d-2)/d}\la^\eps.
		\]
	\end{proof}

	Proposition~\ref{prop:fibration} is a preliminary estimate. To obtain the exponent in Theorem~\ref{thm1}, we refine the estimate for $T_{d,q}(\la,L;C)$ by establishing a recurrence  in $q$.
	\begin{proposition}\label{prop:local}
		We have
		\[
		T_{d,q}(\la,L;C)\ls L^{q-1+2^{-q}}.
		\]
		
	\end{proposition}

	\begin{proof}
		The case $q=0$ is immediate. Indeed, a Euclidean ball of fixed radius contains only $O(1)$ lattice points,
		so $T_{d,0}(\lambda,L;C)\ls 1$.
		
		We now assume $q\ge1$.
		Let $P\subset \Z^d\cap\lambda S^{d-1}$ be admissible for
		$T_{d,q}(\lambda,L;C)$, and let $N=\#P$.  For $h\in\Z^d$, define
		\[
		R(h)=\#\{x\in P:x+h\in P\}.
		\]
		Every ordered pair $(x,y)\in P\times P$ has a unique difference
		$h=y-x$.  Equivalently, writing $y=x+h$, we have
		\[
		N^2=\sum_{h\in\Z^d}R(h).
		\]
		If $R(h)>0$, then $h$ is the difference of two points of $P$. Hence the
		admissibility assumptions imply
		\[
		|\pi_Uh|\le L,
		\qquad
		|\pi_{U^\perp}h|\le C.
		\]
		By Lemma~\ref{lem:rotated-box}, the number of integer vectors satisfying
		these two inequalities is $O(L^q)$.
		
		We separate bounded and unbounded differences.  Fix a constant
		$H\ge 4C+10$.  The contribution of small  $|h|\le H$ is
		$O(N)$, since there are only $O(1)$ such integer vectors and
		$R(h)\le N$ for each $h$.
		
		Now suppose that $R(h)>0$ and large $|h|>H$.  From
		$|\pi_{U^\perp}h|\le C$ and $H\ge4C$ we have
		\[
		|\pi_Uh|
		=\bigl(|h|^2-|\pi_{U^\perp}h|^2\bigr)^{1/2}
		\ge |h|/2.
		\]
		In particular $\pi_Uh\ne0$.  Set
		\[
		\eta=\frac{\pi_Uh}{|\pi_Uh|}\in U.
		\]
		For every $x$ counted by $R(h)$, both $x$ and $x+h$ lie on
		$\lambda S^{d-1}$. Therefore
		\[
		|x+h|^2=|x|^2
		\quad\Longrightarrow\quad
		2h\cdot x+|h|^2=0,
		\]
		or
		\begin{equation}\label{eq:large-h-sphere-relation}
			h\cdot x=-\frac{|h|^2}{2}.
		\end{equation}
		Writing $x=x_U+x_\perp$, with $x_U=\pi_Ux$ and
		$x_\perp=\pi_{U^\perp}x$, equation
		\eqref{eq:large-h-sphere-relation} gives
		\begin{equation}\label{eq:eta-coordinate-control}
			(\pi_Uh)\cdot x_U
			=-\frac{|h|^2}{2}-(\pi_{U^\perp}h)\cdot x_\perp.
		\end{equation}
		If $x,x'$ are both counted by $R(h)$, subtracting
		\eqref{eq:eta-coordinate-control} from $x$ and $x'$ yields
		\[
		|\pi_Uh|\,|\eta\cdot(x_U-x'_U)|
		\le |\pi_{U^\perp}h|\,|x_\perp-x'_\perp|
		\le C^2.
		\]
		Since $|\pi_Uh|\ge H/2$, the scalar coordinate
		$\eta\cdot x_U$ has diameter $O_C(1)$ on the set
		\[
		P_h=\{x\in P:x+h\in P\}.
		\]
		The remaining component of $x_U$ in $U\cap\eta^\perp$ has diameter
		$\le L$, as $\operatorname{diam}(\pi_UP)\le L$.  The component in
		$U^\perp$ is of diameter at most $C$.  Hence $P_h$ is admissible for
		$T_{d,q-1}(\lambda,L;C')$, with respect to the orthogonal decomposition
		\[
		\R^d=(U\cap\eta^\perp)\oplus (U^\perp\oplus\R\eta),
		\]
		where $C'=C'(d,C)$. Thus, every large  $|h|$ satisfies
		\[
		R(h)=\#P_h\le T_{d,q-1}(\lambda,L;C').
		\]
		Combining the estimates for small and large $|h|$ gives
		\[
		N^2
		\ls N+L^qT_{d,q-1}(\lambda,L;C').
		\]
		Consequently,
		\[
		N\ls1+L^{q/2}T_{d,q-1}(\lambda,L;C')^{1/2}.
		\]
		Taking the supremum over all admissible $P$ gives the recurrence
		\begin{equation}\label{eq:T-recurrence-expanded}
			T_{d,q}(\lambda,L;C)
			\ls1+L^{q/2}T_{d,q-1}(\lambda,L;C')^{1/2}.
		\end{equation}
		Because $q\le d-1$, the constants produced by iterating
		\eqref{eq:T-recurrence-expanded} depend only on $d$ and $C$.
		
		Starting from $T_{d,0}\ls1$, define the exponents $\delta_q$ by
		$\delta_0=0$ and
		\[
		\delta_q=\frac q2+\frac{\delta_{q-1}}2.
		\]
		The recurrence gives $T_{d,q}(\lambda,L;C)\ls L^{\delta_q}$, since
		$L\ge1$.  The explicit solution is
		\[
		\delta_q=q-1+2^{-q}.
		\]
		Indeed, this is true for $q=1$, and if
		$\delta_{q-1}=q-2+2^{-(q-1)}$, then
		\[
		\delta_q=\frac q2+\frac{q-2}{2}+2^{-q}=q-1+2^{-q}.
		\]
		This completes the proof.
	\end{proof}
	\begin{remark}{\rm For $1\le L\le \la^{1/2}$, we conjecture that the correct bounds should be $T_{d,1}(\la,L;C)\ls 1$ and $$T_{d,q}(\la,L;C)\ls_\eps L^{q-2}\la^\eps,\quad q\ge2,\ \eps>0.$$ These would imply Conjecture \ref{Aconj}, according to the proof of Propositions \ref{prop:geometric-global} and \ref{prop:endpoint-bound}. Cilleruelo-C\'ordoba \cite{CC1992} proved that $$T_{2,1}(\la,\la^\delta;C)\ls_\delta 1,\ \ \delta<1/2.$$ The case $\delta=1/3$ was proved earlier by Jarnik \cite{Jar1926}. Bourgain-Rudnick \cite[Lemma 2.1]{BR2015} proved that $T_{2,1}(\la,\la^{1/2};C)\ls \log\la$. }
		
	\end{remark}
	
	\begin{lemma}\label{lem:transverse-slab}
		Let \(n_1,\ldots,n_k\in S^{d-1}\) such that
		\(|n_1\wedge\cdots\wedge n_k|\ge\nu\). For \(1\le j\le k\), consider the bands
		\[
		B_j=\{x\in\R^d: |n_j\cdot(x-x_{0,j})|\le C_0\}
		\]
		Put
		\(
		V=\spanop(n_1,\ldots,n_k).
		\)
		Then there is a point \(p\in V\), depending on the bands, such that every
		\(x\in B_1\cap\cdots\cap B_k\) satisfies
		\[
		|\pi_Vx-p|\le C,
		\]
		where \(C=C(d,k,\nu,C_0)\).  In particular, if
		\(x,y\in B_1\cap\cdots\cap B_k\), then $	|\pi_V(x-y)|\ls 1.$
	\end{lemma}
	
	\begin{proof} Let $G=(n_i\cdot n_j)_{1\le i,j\le k}$.
		For each \(j\), write
		\[
		\beta_j=n_j\cdot x_{0,j},
		\qquad
		\beta=(\beta_1,\ldots,\beta_k).
		\]
		Since \(\det G\ge\nu^2>0\), one may choose \(t=(t_1,\ldots,t_k)\) so that
		\(
		Gt=\beta,
		\)
		and define
		\[
		p=t_1n_1+\cdots+t_kn_k\in V.
		\]
		Then
		\(
		n_j\cdot p=\beta_j, \forall 1\leq j \leq k.
		\)
		
		We claim that \(\pi_Vx\) stays close to \(p\) for all
		\(x\in B_1\cap\cdots\cap B_k\).  To see this, fix such an \(x\).  Since
		\(n_j\in V\), we have
		\[
		n_j\cdot \pi_Vx=n_j\cdot x.
		\]
		Therefore,
		\[
		|n_j\cdot(\pi_Vx-p)|
		=
		|n_j\cdot x-\beta_j|
		=
		|n_j\cdot(x-x_{0,j})|
		\le C_0.
		\]
		Now write
		\[
		\pi_Vx-p=s_1n_1+\cdots+s_kn_k.
		\]
		If
		\[
		c_j=n_j\cdot(\pi_Vx-p),
		\qquad
		c=(c_1,\ldots,c_k),
		\]
		then $Gs=c$, and
		\[
		|s|\le C(k,\nu)|c|\le C(k,\nu,C_0).
		\]
		Since the vectors \(n_j\) have length \(1\),
		\[
		|\pi_Vx-p|
		=
		|s_1n_1+\cdots+s_kn_k|
		\le C(d,k,\nu,C_0).
		\]
		This proves the first claim. Now if \(x,y\in B_1\cap\cdots\cap B_k\), then
		\[
		|\pi_V(x-y)|
		\le |\pi_Vx-p|+|\pi_Vy-p|
		\le C(d,k,\nu,C_0).
		\]
		Hence \( |\pi_V(x-y)|\ls 1 \).
	\end{proof}
	The following lemma decomposes large spherical regions into a number of standard bricks that can be handled by Propositions \ref{prop:fibration} and \ref{prop:local}.
	\begin{lemma}\label{lem:shell}
		Let \(\R^d=V\oplus W\), with \(m=\dim W\ge2\). Suppose
		\[
		\Omega\subset\{x\in\la S^{d-1}:|\pi_Vx-p|\le C\}
		\]
		for some \(p\in V\).  Let \(R=(\max(\la^2-|p|^2,0))^{1/2}\).
		Then one of the following holds.
		\begin{enumerate}
			\item[(i)] If \(R\ls\la^{1/2}\), then
			\(\operatorname{diam}(\pi_W\Omega)\ls\la^{1/2}\).
			\item[(ii)] If \(R\gg \la^{1/2}\), then \(\Omega\) is covered by
			\(O(\la^{\frac{m-1}2})\) many sets \(\Omega_\alpha\), each satisfying
			\[
			\operatorname{diam}(\pi_{V\oplus\R\omega_\alpha}\Omega_\alpha)\ls1,
			\qquad
			\operatorname{diam}(\pi_{W\cap\omega_\alpha^\perp}\Omega_\alpha)
			\ls\la^{1/2}
			\]
			for some unit vector \(\omega_\alpha\in W\).
		\end{enumerate}
	\end{lemma}

	\begin{proof}
		Write every point $x\in\Omega$ as
		\[
		x=p+v+w,
		\qquad v\in V,
		\qquad w\in W,
		\qquad |v|\le C.
		\]
		Since $V\perp W$ and $|x|=\lambda$, we have
		\begin{equation}\label{eq:shell-basic-equation}
			|w|^2=\lambda^2-|p+v|^2.
		\end{equation}
		Therefore,
		\[
		\bigl||p+v|^2-|p|^2\bigr|
		\le 2|p|\,|v|+|v|^2
		\ls_C \lambda.
		\]
		By definition,
		\[
		R^2=\max(\lambda^2-|p|^2,0),
		\]
		this implies the uniform shell estimate
		\begin{equation}\label{eq:shell-square-thickness}
			\bigl||w|^2-R^2\bigr|\ls\lambda.
		\end{equation}
		Indeed, if $|p|\le\lambda$, then
		$R^2=\lambda^2-|p|^2$ and the conclusion follows directly from
		\eqref{eq:shell-basic-equation}.  If $|p|>\lambda$, then $R=0$ and the same
		bound follows because $\lambda^2-|p|^2<0$ while
		$|w|^2=\lambda^2-|p+v|^2\ge0$ differs from
		$\lambda^2-|p|^2$ by $O(\lambda)$.
		
		If $R\ls\lambda^{1/2}$, then \eqref{eq:shell-square-thickness} gives
		$|w|^2\ls\lambda$, hence $|w|\ls\lambda^{1/2}$ for all $x\in\Omega$.
		Thus,
		\[
		\operatorname{diam}(\pi_W\Omega)
		\le 2\sup_{x\in\Omega}|\pi_Wx|
		\ls\lambda^{1/2}.
		\]
		This proves (i).
		
		We now assume $R\gg\lambda^{1/2}$.  From
		\eqref{eq:shell-square-thickness},
		\[
		\bigl||w|-R\bigr|
		=\frac{\bigl||w|^2-R^2\bigr|}{|w|+R}
		\ls \frac{\lambda}{R}.
		\]
		Thus, $\pi_W\Omega$ lies in a radial shell around the sphere
		$\{w\in W:|w|=R\}$ of thickness $O(1+\lambda/R)$.
		
		Put $q=m-1=\dim W-1$.  Choose a maximal $c\lambda^{1/2}$-separated set on
		the sphere $\{w\in W:|w|=R\}$, where $c>0$ is a small fixed constant. The
		corresponding spherical caps of tangential radius $O(\lambda^{1/2})$ cover
		the sphere, and the number of such caps is
		$	O((\frac{R}{\lambda^{1/2}})^q),$
		because $R\gg\lambda^{1/2}$.  Let $R\omega$ be the center of one cap
		with $\omega\in W$ and $|\omega|=1$.
		
		Consider the portion of the shell that lies over this cap.  For such a point
		$w$, write
		\[
		w=(\omega\cdot w)\omega+w_\perp,
		\qquad w_\perp\in W\cap\omega^\perp.
		\]
		The angular localization gives
		\[
		|w_\perp|\ls\lambda^{1/2}.
		\]
		We will then estimate the range of the scalar coordinate $\omega\cdot w$.  Since
		$w$ is in a small cap, $\omega\cdot w$ is comparable to $R$, and
		\[
		|w|-\omega\cdot w
		=\frac{|w_\perp|^2}{|w|+\omega\cdot w}
		\ls \frac{\lambda}{R}.
		\]
		The radial variable $|w|$ itself varies over an interval of length
		$O(1+\lambda/R)$.  Hence $\omega\cdot w$ also varies over an interval of
		length $O(1+\lambda/R)$
		on this cap.  We partition that interval into subintervals of bounded
		length.  This introduces $O(1+\lambda/R)$ many subpieces for each angular cap.
		
		On each resulting piece $\Omega_\alpha$, the $V$-projection has bounded
		diameter as $|\pi_Vx-p|\le C$, and the $\omega$-coordinate has bounded
		diameter by construction. Therefore,
		\[
		\operatorname{diam}(\pi_{V\oplus\R\omega}\Omega_\alpha)\ls1.
		\]
		At the same time the transverse component in $W\cap\omega^\perp$ has
		diameter $O(\lambda^{1/2})$, so
		\[
		\operatorname{diam}(\pi_{W\cap\omega^\perp}\Omega_\alpha)
		\ls\lambda^{1/2}.
		\]
		See Figure \ref{fig1}.
		\begin{figure}[htbp]
			\centering
			\includegraphics[width=0.6\textwidth]{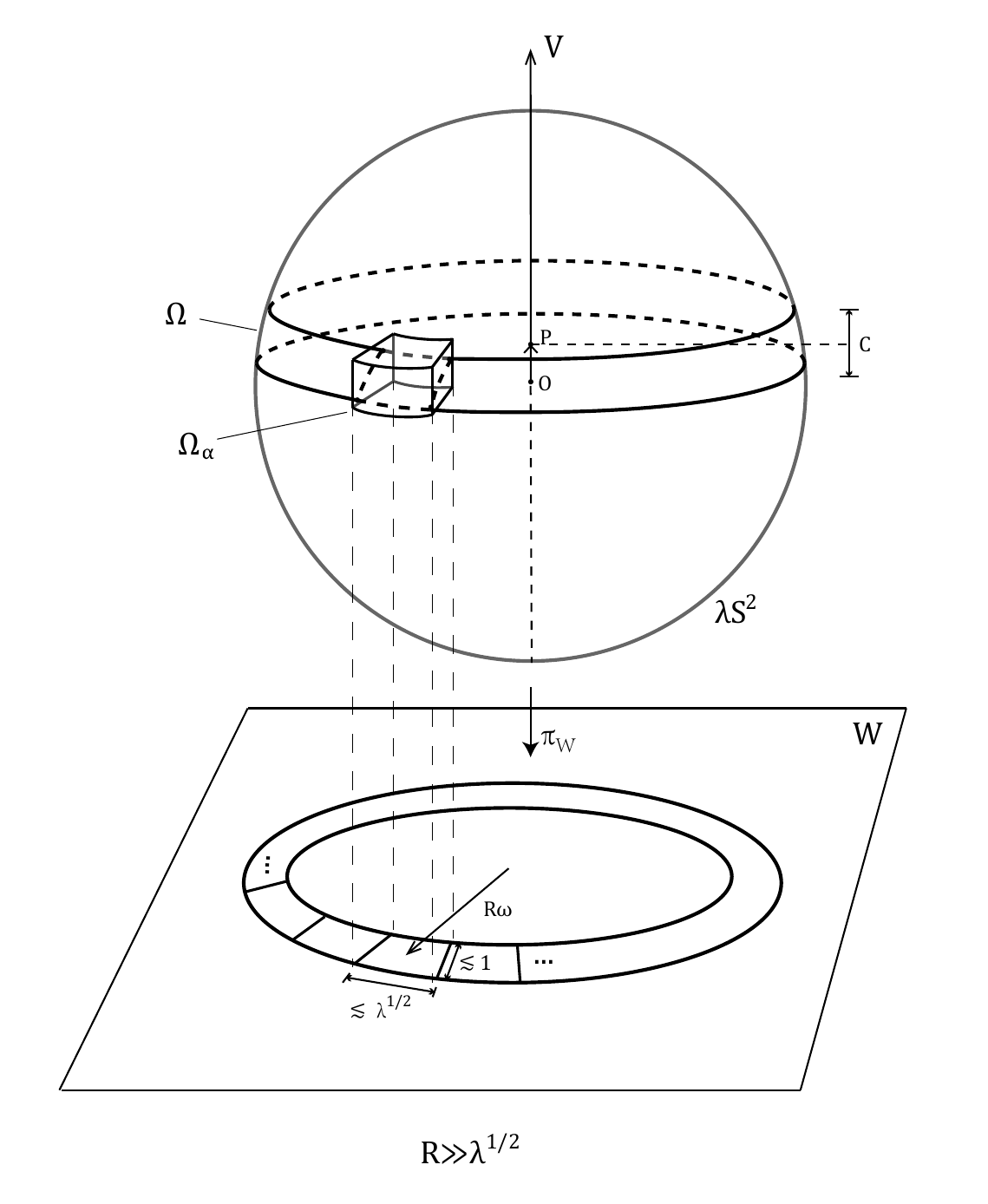}
			\caption{Lemma \ref{lem:shell} (ii)}
			\label{fig1}
		\end{figure}
		
		It remains only to count the pieces.  Since $R\le\lambda+C\ls\lambda$,
		we have
		\[
		\Big(\frac{R}{\lambda^{1/2}}\Big)^q
		\Big(1+\frac{\lambda}{R}\Big)
		\ls
		\lambda^{q/2}.
		\]
		Indeed, the first term is at most $\lambda^{q/2}$, and the second is
		\[
		\Big(\frac{R}{\lambda^{1/2}}\Big)^q\frac{\lambda}{R}
		=R^{q-1}\lambda^{1-q/2}
		\ls\lambda^{q/2}.
		\]
		Thus $\Omega$ is covered by $O(\lambda^{q/2})=O(\lambda^{(m-1)/2})$
		pieces with the asserted projection bounds.  This proves (ii).
	\end{proof}
	
	\begin{proposition}\label{prop:geometric-global}
		Let \(d\ge3\) and \(m=d-k\ge2\).  Then, for every \(\eps>0\),
		\[
		A_{k,d,\la}\ls_\eps\la^{\beta(m,d)+\eps},
		\]
		with 
		\begin{equation}\label{eq:Egeom}
			\beta(m,d)=
			\min\left\{\frac{(m-1)(d-1)}d,\ m-\frac32+2^{-m}\right\}.
		\end{equation}
	\end{proposition}
	
	\begin{proof}
		Let $P\subset \Z^d\cap\lambda S^{d-1}$ be as before. Again, put
		\(V=\spanop(n_1,\ldots,n_k)\), \(W=V^\perp\), \(m=d-k\).  If
		\(P\ne\varnothing\), choose \(x_0\in P\), and set \(p=\pi_Vx_0\).
		Lemma~\ref{lem:transverse-slab} gives \(|\pi_Vx-p|\ls1\) for all \(x\in P\), so
		we may apply Lemma~\ref{lem:shell}. Propositions \ref{prop:fibration} and \ref{prop:local} gives
		\[T_{d,q}(\la,\la^{1/2};C)\ls \la^{\gamma(q,d)+\eps}\]
		where 
		\[\gamma(q,d)=\min\left\{\frac{q(d-2)}d,\ q-1+2^{-q}\right\}.\]
		
		In case (i), \(P\) is admissible for
		\(T_{d,m}(\la,\la^{1/2};C)\), hence
		\[
		\#P\ls_\eps \la^{\gamma(m,d)/2+\eps}.
		\]
		This exponent is no larger than both exponents in
		\eqref{eq:Egeom}.
		
		In case (ii), \(P\) is covered by \(O(\la^{\frac{m-1}2})\) patches that are admissible for
		\(T_{d,m-1}(\la,\la^{1/2};C)\). Therefore,
		\[
		\#P\ls_\eps \la^{\frac{m-1}2}\la^{\gamma(m-1,d)/2}\la^\eps
		=\la^{\beta(m,d)+\eps}.
		\]
	\end{proof}
	
	\begin{lemma}\label{lem:end-tube}
		Let \(B_1,\ldots,B_{d-1}\) be the transverse bands as in
		Lemma~\ref{lem:transverse-slab}.  Then
		\[
		B_1\cap\cdots\cap B_{d-1}
		\]
		is contained in a fixed-radius tube around some line \(\ell\). Moreover, the part of this tube lying on \(\la S^{d-1}\) can be divided into at
		most two pieces.  On each piece, the projection to the direction of \(\ell\) has
		diameter \(O(\la^{1/2})\).
	\end{lemma}
	
	\begin{proof}
		Let \(n_1,\ldots,n_{d-1}\) be the unit normals of the bands, and put
		\( V=\spanop(n_1,\ldots,n_{d-1}) \).
		Then by Lemma~\ref{lem:transverse-slab}, the band intersection is contained in a fixed-radius tube around a line \(\ell\) parallel to \(V^\perp\).  
		
		It remains to describe the intersection of this tube with the sphere.  Write
		\[
		\ell=p+\R v,
		\qquad |v|=1,
		\qquad p\perp v.
		\]
		Every point in the tube can be written as
		\[
		x=p+tv+y,
		\qquad y\perp v,
		\qquad |y|\le C,
		\]
		where \(C\) is fixed.  If also \(|x|=\la\), then
		\(
		t^2+|p+y|^2=\la^2.
		\)
		If the intersection is nonempty, then \(|p|\le \la+C\).  Hence, as \(y\) ranges
		over the bounded set \(|y|\le C\), the quantity \(|p+y|^2\) varies by
		\(O(\la)\).  Therefore \(t^2\) also varies by \(O(\la)\).
		
		Now split the spherical part of the tube into the two pieces \(t\ge0\) and
		\(t\le0\).  On either piece,
		\[
		|t_1-t_2|
		\le
		|t_1^2-t_2^2|^{1/2}
		\ls \la^{1/2}.
		\]
		Thus the projection to the direction of \(\ell\) has diameter
		\(O(\la^{1/2})\) on each of the two pieces.
	\end{proof}
	
	\begin{proposition}\label{prop:endpoint-bound}
		For \(d\ge3\) and every \(\eps>0\),
		\[
		A_{d-1,d,\la}\ls_\eps\la^{\beta(1,d)+\eps},
		\]
		with 
		\begin{equation}\label{eq:endpoint-exponent}
			\beta(1,d)=\min\left\{\frac{d-2}{2d},\frac14\right\}.
		\end{equation}
	\end{proposition}
	
	\begin{proof}
		By Lemma~\ref{lem:end-tube}, the set is covered by two admissible sets for
		\(T_{d,1}(\la,\la^{1/2};C)\).  Proposition~\ref{prop:local} gives
		\[
		T_{d,1}(\la,\la^{1/2};C)
		\ls_\eps (\la^{1/2})^{\min\{(d-2)/d,1/2\}}\la^\eps
		=\la^{\beta(1,d)+\eps},
		\]
		which is the asserted exponent.
	\end{proof}
	
	\section{Discrete spherical multiplier}\label{sec4}
	
	In this section, we use an approximation of the discrete spherical multiplier to establish new restriction estimates for eigenfunctions. 
	
	Let $n=\la^2\in \mathbb{N}$. Let $\E_\la=\{\xi\in \mathbb{Z}^d:|\xi|=\la\}$. We write the discrete spherical multiplier
	\[
	S_n(x)=\sum_{\xi\in \E_\la}e^{2\pi i x\cdot \xi}
	\]
	We shall use the approximation of the discrete spherical multiplier $S_n(x)$ from Magyar-Stein-Wainger \cite[Section 3]{msw} and Magyar \cite[Lemma 1]{m}, involving Kloosterman sums from analytic number theory. For any integer $q\ge1$, write $q=2^h q_1$ for some integer $h\ge0$ and odd integer $q_1\ge1$. Define
	\[
	\rho(q,n)=2^{h}(q_1,n),
	\]
	where $(q_1,n)$ denotes the greatest common divisor of $q_1$ and $n$.
	Let \(\sigma\) denote normalized Euclidean surface measure on
	\(S^{d-1}\). Its Fourier transform has the decay estimate
	\begin{equation}
		|\hat\sigma(x)|\ls (1+|x|)^{-\frac{d-1}2}.
	\end{equation}
	We use the dimensional constant
	\begin{equation}\label{eq:Cd-definition}
		C_d=\frac12|S^{d-1}|=\frac{\pi^{d/2}}{\Gamma(d/2)}.
	\end{equation}

	\begin{proposition}[Magyar-Stein-Wainger \cite{msw} and Magyar \cite{m}]\label{propm}
		Let \(d\ge4\) and \(n=\la^2\in\N\).  Then
		the discrete spherical multiplier admits the decomposition
		\begin{equation}
			S_n(x)=C_d\la^{d-2}M_n(x)+E_n(x),
		\end{equation}
		where
		\begin{equation}
			M_n(x)=
			\sum_{1\le q\le\la}\sum_{m\in\Z^d}
			K_q(m,n)\psi(qx-m)\widehat\sigma(\la(x-m/q)).
		\end{equation}
		Here
		\begin{equation}
			K_q(m,n)=q^{-d}
			\sum_{a\in(\Z/q\Z)^\times}
			\sum_{r\in(\Z/q\Z)^d}e^{2\pi i(a(|r|^2-n)+m\cdot r)/q}.
		\end{equation}
		For every \(\eps>0\),
		\begin{equation}\label{Ksum}	|K_q(m,n)|\ls_{\eps}q^{-\frac{d-1}2+\eps}\rho(q,n)^{\frac12}
		\end{equation}
		and
		\begin{equation}
			|E_n(x)|\ls_{\eps} \la^{\frac{d-1}2+\eps}.
		\end{equation}
		Here 
		\(\psi\in C_0^\infty(\R^d)\) is supported on \(\max_j|x_j|\le1/4\) and equal to 1 on \(\max_j|x_j|\le1/8\).
	\end{proposition}
	The estimate \eqref{Ksum} is obtained from Weil's bound \cite[Chapter 4]{s1990} for a Kloosterman sum or Sali\'e sum, depending on whether $n$ is even or odd. Thus, for $s=\frac{d-1}2$ and every \(\eps>0\),
	\begin{equation}\label{sn}
		|S_n(x)|
		\ls_{\eps}\la^{s+\eps}+\la^{d-2}
		\sum_{1\le q\le\la}q^{-s+\eps}\rho(q,n)^{1/2}
		\sum_{m\in\Z^d}
		|\psi(qx-m)|
		\bigl(1+\la|x-m/q|\bigr)^{-s}.
	\end{equation}

	\subsection{Proof of restriction estimates}
	Let $\Sigma\subset \mathbb{T}^d$ be a compact smooth $k$-dimensional submanifold. By a finite partition of unity, it is enough to work on a coordinate patch. Thus, let 
	\[\Gamma: U\subset \mathbb{R}^k\to \mathbb{R}^d\]
	be a smooth lift of a patch of $\Sigma$. After shrinking the patch, we may assume that $\Gamma$ is bi-Lipschitz:
	\begin{equation}\label{bl}
		|\Gamma(u)-\Gamma(v)|\approx |u-v|,\ \ u,v\in U.
	\end{equation}
	We define $T_\la:\ell^2(\E_\la)\to L^2(U)$ as
	\[(T_\la a)(u)=\sum_{\xi\in \E_\la}a_\xi e^{2\pi i\xi\cdot \Gamma(u)}.\]
	By a $TT^*$ argument and Schur's test, it suffices to estimate the integral
	\begin{equation}\label{schur}
		\sup_{u\in U}\int_U|S_n(\Gamma(u)-\Gamma(v))|dv.
	\end{equation}
	Then the norm of $T_\la$ will be bounded by the square root of this integral estimate.

	

	\begin{lemma}\label{lem:packing}
		For \(1\le q\le\la\), let
		\[
		I_q(u)=\int_{U}
		\sum_{m\in\Z^d}
		|\psi(q(\Gamma(u)-\Gamma(v))-m)|
		\bigl(1+\la|\Gamma(u)-\Gamma(v)-m/q|\bigr)^{-s}d v.
		\]
		Then we have
		\begin{equation}
			I_q(u)
			\ls
			\begin{cases}
				q^k\la^{-k},&s>k,\\
				q^k\la^{-k}\log(2\la/q),&s=k,\\
				q^s\la^{-s},&0<s<k.
			\end{cases}
		\end{equation}
	\end{lemma}

	\begin{proof}
		The support of $\psi$ forces $m$ to lie in a fixed-radius neighborhood of the $k$-dimensional set $q(\Gamma(u)-\Gamma(U))$. This neighborhood contains $O(q^k)$ lattice points. For each such $m$, the set of $v$ contributing to the integral is contained, up to finite multiplicity, in a ball of radius $O(1/q)$ in the parameter variables. By \eqref{bl}, the integral is bounded by
		\[
		\int_{z\in \mathbb{R}^k:|z|\ls 1/q}(1+\lambda |z|)^{-s}\,dz
		\ls
		\begin{cases}
			\lambda^{-k},&s>k,\\
			\lambda^{-k}\log(2\lambda/q),&s=k,\\
			\lambda^{-s}q^{s-k},&0<s<k.
		\end{cases}
		\]
		Multiplying by the $O(q^k)$ possible values of $m$ gives the desired estimate.
	\end{proof}
	\begin{lemma}\label{lem:rho-absorption}
		Let \(n=\la^2\in\mathbb{N}\), \(a\in\R\).  For every
		\(\eps>0\),
		\[
		\sum_{q\le \la}q^a\rho(q,n)^{1/2}
		\ls_{\eps} \la^{\max(a+1,0)+\eps}.
		\]
	\end{lemma}
	
	\begin{proof}
		Write \(q=2^h q_1\), where \(q_1\) is odd.  Since
		\[
		(q_1,n)^{1/2}\le
		\sum_{\substack{\delta\mid n\\ \delta\mid q_1}}\delta^{1/2},
		\]
		we have
		\[
		\sum_{q\le \la}q^a\rho(q,n)^{1/2}
		\le
		\sum_{2^h\le \la}
		\sum_{\substack{\delta\mid n\\ \delta\ {\rm odd}}}
		2^{h/2}\delta^{1/2}(2^h\delta)^a
		\sum_{r\le \la/(2^h\delta)}r^a.
		\]
		If \(a>-1\), then the inner sum is 
		\[\sum_{r\le \la/(2^h\delta)}r^a\ls (\la/(2^h\delta))^{a+1}\] and then
		\[
		\sum_{q\le \la}q^a\rho(q,n)^{1/2}\ls \la^{a+1}\sum_{2^h\le \la}2^{-h/2}
		\sum_{\delta\mid n}\delta^{-1/2}
		\ls_{\eps}\la^{a+1+\eps}.
		\]
		The factor $\la^\eps$ comes from the divisor bound.
		If \(a\le-1\), then the inner sum is \(O_\eps( (\la/(2^h\delta))^{\eps})\) for any $\eps>0$, and the remaining sums are
		\(O_{\eps}(\la^\eps)\). 
	\end{proof}
	
	\begin{remark}
		The factor \(\rho(q,n)^{1/2}\) does not worsen the power of \(\la\) beyond the
		usual unweighted \(q^a\)-summation, apart from an \(\la^\eps\) loss.
	\end{remark}
	\begin{proposition}
		For every
		\(\eps>0\),	\[	\sup_{u\in U}\int_U|S_n(\Gamma(u)-\Gamma(v))|dv\ls_\eps	\la^{d-k-2+\eps}+\la^{\frac{d-1}2+\eps}. \]
	\end{proposition}
	\begin{proof}
		Let \(s=\frac{d-1}2\). By \eqref{sn}, we obtain
		\[
		\begin{aligned}
			\int_U|S_n(\Gamma(u)-\Gamma(v))|dv&\ls_\eps
			\la^{d-2}
			\sum_{1\le q\le\la}q^{-s+\eps}\rho(q,n)^{1/2}I_{q}(u)
			+\la^{s+\eps}.
		\end{aligned}
		\]

		If \(s>k\), Lemma~\ref{lem:packing} gives \(I_{q}(u)\ls q^k\la^{-k}\), so Lemma~\ref{lem:rho-absorption} implies 
		\begin{align*}
			\int_U|S_n(\Gamma(u)-\Gamma(v))|dv&\ls_\eps
			\la^{d-k-2}
			\sum_{q\le\la}q^{k-s+\eps}\rho(q,n)^{1/2}
			+\la^{s+\eps}\\
			&\ls_\eps
			\la^{d-k-2+\max(k-s+1,0)+\eps}+\la^{s+\eps}.
		\end{align*}
		If \(k\le s-1\), this is \(\la^{d-k-2+\eps}+\la^{s+\eps}\). If
		\(s-1<k<s\), the bound is  $\la^{s+\eps}$.
		
		If \(s=k\), then Lemma~\ref{lem:packing} gives
		$$I_{q}(u)\ls q^k\la^{-k}\log(2\la/q)$$ and Lemma~\ref{lem:rho-absorption}
		again gives the bound \(\la^{s+\eps}\). 
		
		If \(s<k\), then
		Lemma~\ref{lem:packing} gives \(I_{q}(u)\ls q^s\la^{-s}\), and Lemma~\ref{lem:rho-absorption}
		also implies the bound \(\la^{s+\eps}\).
		These bounds are uniform in $u$. Combining the cases yields the result.
	\end{proof}
	Using the estimate for the norm of $T_\la$, we immediately obtain the following new restriction estimates for eigenfunctions.
	\begin{theorem}\label{ntthm}
		Let \(d\ge4\) and $1\le k\le d-1$.  Let $\Sigma\subset \mathbb{T}^d$ be a compact smooth $k$-dimensional submanifold. Then for any
		\(\eps>0\),
		\[\|e_\la\|_{L^2(\Sigma)}\ls_\eps (\la^{\frac{d-k-2}2+\eps}+\la^{\frac{d-1}4+\eps})\|e_\la\|_{L^2(\mathbb{T}^d)}.\]
		In particular, when $k\le \frac{d-3}2$ we have the sharp estimate 
		\[\|e_\la\|_{L^2(\Sigma)}\ls_\eps \la^{\frac{d-k-2}2+\eps}\|e_\la\|_{L^2(\mathbb{T}^d)}.\]  
	\end{theorem}
	Since $\frac{d-1}4<\frac{d-k-1}2$ when $k<\frac{d-1}2$, these bounds improve \eqref{BGT01} when $k<\frac{d-1}2$. We remark that if $\Sigma$ is totally geodesic, then the implicit constants in Theorem \ref{ntthm} can be independent of the direction of $\Sigma$. Therefore, by Theorem \ref{prop3a}, the following estimates on $A_{k,d,\la}$ hold.
	\begin{proposition}\label{prop:circle-band}
		Let \(d\ge4\) and $1\le k\le d-1$.  Then, for any
		\(\eps>0\),
		\[
		A_{k,d,\la}\ls_{\eps}
		\la^{d-k-2+\eps}+\la^{\frac{d-1}2+\eps}.
		\]	In particular, when $k\le \frac{d-3}2$ we have the sharp estimate 
		\[A_{k,d,\la}\ls_\eps \la^{d-k-2+\eps}.\]  
	\end{proposition}

	\section{Proof of Theorem~\ref{thm1}}\label{sec5}
	Let $m=d-k\ge1$ be the codimension of the totally geodesic submanifold. By
	Theorem~\ref{prop:logfree-reduction}, the restriction estimates follow from the
	corresponding bounds for $A_{k,d,\lambda}$ in Propositions \ref{prop:endpoint-bound}, \ref{prop:geometric-global}, \ref{prop:circle-band}.
	
	If $m=1$, then Proposition~\ref{prop:endpoint-bound} gives the
	lattice exponent
	\[
	\beta(1,d)=\min\left\{\frac{d-2}{2d},\frac14\right\}.
	\]
	Dividing by two gives $\alpha(1,3)=1/12$ and $\alpha(1,d)=1/8$ for $d\ge4$.
	
	Assume next that $2\le m\le d-1$. Proposition~\ref{prop:geometric-global}
	gives the lattice exponent
	\[
	\beta(m,d)=
	\min\left\{\frac{(m-1)(d-1)}d,\ m-\frac32+2^{-m}\right\}.
	\]
	For $d\ge5$, Proposition~\ref{prop:circle-band} also gives
	\[
	A_{d-m,d,\lambda}
	\ls \lambda^{m-2+\eps}+\lambda^{(d-1)/2+\eps}.
	\]
	This has the conjectural lattice exponent $m-2$ exactly when
	$m-2\ge (d-1)/2$, equivalently $d\le 2m-3$.
	
	For $m=2,3$,  we have $\alpha(m,d)=\beta(m,d)/2$. For
	$m\ge4$, Proposition~\ref{prop:circle-band} gives $\alpha(m,d)=m/2-1$ when $d\le 2m-3$.
	Outside that range the exponent $\beta(m,d)$ is smaller: at $d=2m-2$ the first
	term in $\beta(m,d)$ equals $m-3/2$, while for $d\ge2m-1$ the second
	term equals $m-3/2+2^{-m}$ and is the smaller one. Dividing these lattice
	exponents by two gives the stated formula for $\alpha(m,d)$. 
	In particular, Theorem \ref{ntthm} gives sharp estimates for any compact submanifold of codimension $m\ge \frac{d+3}2$.

	\section{Further discussions on restrictions to curves}\label{sec6}
	In this section, we prove some new restriction estimates for curves and discuss their relation to the Discrete Restriction Conjecture (Bourgain-Demeter \cite[Conjecture 2.6]{BD}). By Theorem \ref{ntthm}, for any curve segment $\Sigma \subset \mathbb{T}^d$, when $d\ge5$ we have sharp estimates 
	\begin{equation}\label{T5}
		\|e_\la\|_{L^2(\Sigma)}\ls_\eps \la^{\frac{d-3}2+\eps}\|e_\la\|_{L^2(\mathbb{T}^d)},
	\end{equation}
	and when $d=4$, we have
	\begin{equation}\label{T4}
		\|e_\la\|_{L^2(\Sigma)}\ls \la^{\frac34+\eps}\|e_\la\|_{L^2(\mathbb{T}^4)}.
	\end{equation}
	We remark that the power $\frac34$ agrees with the one from the slicing and packing method in Proposition \ref{prop:geometric-global}. The conjectural power is $\frac12$. See Conjecture \ref{HZconj} and Remark \ref{dcrem}.

	When $d=2$, Bourgain-Rudnick proved the following result.
	\begin{theorem}[Bourgain-Rudnick \cite{BR2011,BR2015}]
		If $\Sigma\subset \mathbb{T}^2$ is a geodesic segment, then 
		\begin{equation}
			\|e_\la\|_{L^2(\Sigma)}\ls \sqrt{\log\la}\|e_\la\|_{L^2(\mathbb{T}^2)}.
		\end{equation}
		If $\Sigma\subset \mathbb{T}^2$ is a smooth curve segment whose geodesic curvature is nowhere vanishing, then 
		\begin{equation}\label{unf}
			\|e_\la\|_{L^2(\Sigma)}\ls \|e_\la\|_{L^2(\mathbb{T}^2)}.
		\end{equation}
	\end{theorem}
	It is conjectured that the uniform bound should hold for any real-analytic curve. Bourgain-Rudnick's proof in \cite{BR2011} actually gives the uniform bound \eqref{unf} for any real-analytic curve that is not a geodesic. Thus geodesics are the only unsolved case for \eqref{unf} among all real-analytic curves.
	
	When $d=3$, we prove the following results using Bourgain--Rudnick's cap estimates and the van der Corput lemma. These improve the general bound $O(\la^\frac12\sqrt{\log\la})$ in \eqref{BGT01}, while the conjectural bound is $O_\eps(\la^\eps)$. See Conjecture \ref{HZconj} and Remark \ref{dcrem}.
	\begin{theorem}\label{T3thm}
		If $\Sigma\subset \mathbb{T}^3$ is (i) a geodesic segment, or (ii) a smooth
		curve segment whose torsion is nowhere vanishing, or (iii) a real-analytic
		curve segment whose geodesic curvature is nowhere vanishing, then for any $\eps>0$
		\begin{equation}\label{T3}
			\|e_\la\|_{L^2(\Sigma)}\ls_\eps \la^{\frac13+\eps}\|e_\la\|_{L^2(\mathbb{T}^3)}.
		\end{equation}
		If $\Sigma\subset \mathbb{T}^3$ is (iv) a smooth planar curve segment whose
		geodesic curvature is nowhere vanishing, then 
		\begin{equation}
			\|e_\la\|_{L^2(\Sigma)}\ls_\eps \la^{\frac14+\eps} \|e_\la\|_{L^2(\mathbb{T}^3)}.
		\end{equation}
	\end{theorem}
	\begin{proof}
		The geodesic case (i) follows from Theorem \ref{thm1}.  We prove the
		remaining cases by a local \(TT^*\) argument.  Write
		\[
		e_\lambda(x)=\sum_{\xi\in \mathcal E_\lambda}a_\xi e^{2\pi i\xi\cdot x},
		\qquad
		\mathcal E_\lambda=\mathbb Z^3\cap \lambda S^2,
		\qquad
		\|e_\lambda\|_{L^2(\mathbb T^3)}^2=\sum_{\xi\in\mathcal E_\lambda}|a_\xi|^2 .
		\]
		Choose a lift and an arclength parametrization
		\[
		\gamma:I\to\mathbb R^3,\qquad |\gamma'(t)|=1 .
		\]
		After a finite partition of unity it suffices to prove a weighted estimate
		with \(0\leq \chi\in C_0^\infty(I)\).  Put
		\[
		K(h)=\int_I \chi(t)e^{2\pi i h\cdot \gamma(t)}\,dt .
		\]
		Then
		\[
		\int_I \left|\sum_{\xi\in\mathcal E_\lambda}
		a_\xi e^{2\pi i\xi\cdot\gamma(t)}\right|^2\chi(t)\,dt
		=
		\sum_{\xi,\eta\in\mathcal E_\lambda}
		a_\xi\overline{a_\eta}K(\xi-\eta).
		\]
		Since \(K(-h)=\overline{K(h)}\), Schur's test reduces the proof to bounding
		\[
		B_\lambda:=\sup_{\xi\in\mathcal E_\lambda}
		\sum_{\eta\in\mathcal E_\lambda}|K(\xi-\eta)|.
		\]
		
		We shall use Bourgain--Rudnick's cap estimate \cite{BR2011} in the following two forms. If $F_3(\la,r)$ denotes the maximal number of points of $\mathbb{Z}^3\cap \la S^2$ in a spherical cap of Euclidean diameter $r$, then for \(1\le r\le 2\lambda\),
		\[
		F_3(\la,r)
		\lesssim_\varepsilon \lambda^\varepsilon(1+r).
		\]
		Consequently, for every two-plane \(P\subset\mathbb R^3\), with
		\(\pi_P\) denoting orthogonal projection onto \(P\),
		\[
		N_\xi(U;P)
		:=
		\#\{\eta\in\mathcal E_\lambda:|\pi_P(\eta-\xi)|\le U\}
		\lesssim_\varepsilon
		\lambda^\varepsilon\bigl(1+(\lambda U)^{1/2}\bigr),
		\qquad 1\le U\le 2\lambda .
		\]
		Indeed, the set is contained in the intersection of \(\lambda S^2\) with a
		\(U\)-tube about the affine line \(\xi+P^\perp\), and this intersection is
		covered by \(O(1)\) caps of Euclidean diameter \(C(\lambda U)^{1/2}\).
		
		We also use the standard finite-type van der Corput lemma: if
		\[
		\max_{1\le j\le m} |\phi^{(j)}(t)|\ge A
		\]
		on an interval, with the amplitude in a fixed bounded \(C^1\)-family, then
		\[
		\left|\int e^{i\phi(t)}a(t)\,dt\right|\lesssim A^{-1/m}.
		\]
		Now Case (ii) and (iv) follow immediately.
		
		\noindent\textbf{(ii) Nowhere vanishing torsion case.}
		Assume
		\[
		\det(\gamma',\gamma'',\gamma''')\neq 0
		\]
		on \(\operatorname{supp}\chi\). Then the finite-type van der Corput estimate gives the following.
		\[
		|K(h)|\lesssim (1+|h|)^{-1/3}.
		\]
		Therefore,
		\[
		\begin{aligned}
			B_\lambda
			&\lesssim
			\sum_{\substack{1\le R\le 2\lambda\\ R\ \mathrm{dyadic}}}
			R^{-1/3}\#\{\eta\in\mathcal E_\lambda:|\eta-\xi|\le 2R\}  \\
			&\lesssim_\varepsilon
			\lambda^\varepsilon
			\sum_{\substack{1\le R\le 2\lambda\\ R\ \mathrm{dyadic}}}
			R^{-1/3}(1+R)                                             \\
			&\lesssim_\varepsilon \lambda^{2/3+\varepsilon}.
		\end{aligned}
		\]
		Thus,
		\[
		\|e_\lambda\|_{L^2(\Sigma)}
		\lesssim_\varepsilon
		\lambda^{1/3+\varepsilon}
		\|e_\lambda\|_{L^2(\mathbb T^3)} .
		\]
		
		\noindent\textbf{(iv) Planar curved case.}
		Assume \(\gamma(I)\subset x_0+P\), where \(P\) is a two-dimensional linear
		plane, and assume that the curvature in \(P\) is nowhere zero in
		\(\operatorname{supp}\chi\).  Then
		\[
		|K(h)|\lesssim (1+|\pi_P h|)^{-1/2}.
		\]
		Hence, summing dyadically in \(U=1+|\pi_P(\eta-\xi)|\),
		\[
		\begin{aligned}
			B_\lambda
			&\lesssim
			\sum_{\substack{1\le U\le 2\lambda\\ U\ \mathrm{dyadic}}}
			U^{-1/2}N_\xi(2U;P)                                      \\
			&\lesssim_\varepsilon
			\lambda^\varepsilon
			\sum_{\substack{1\le U\le 2\lambda\\ U\ \mathrm{dyadic}}}
			U^{-1/2}\bigl(1+(\lambda U)^{1/2}\bigr)                    \\
			&\lesssim_\varepsilon \lambda^{1/2+\varepsilon}.
		\end{aligned}
		\]
		And Schur's test gives
		\[
		\|e_\lambda\|_{L^2(\Sigma)}
		\lesssim_\varepsilon
		\lambda^{1/4+\varepsilon}
		\|e_\lambda\|_{L^2(\mathbb T^3)} .
		\]
		
		Finally Case (iii) requires some further discussion.
		
		\noindent\textbf{(iii) Analytic curved case.}
		Assume now that \(\gamma\) is real-analytic and has nowhere vanishing geodesic curvature,
		\[
		|\gamma''(t)|\ge c_0>0.
		\]
		Let
		\[
		\tau(t)=\det(\gamma'(t),\gamma''(t),\gamma'''(t)).
		\]
		If \(\tau\equiv0\) on a subinterval, then the curve is planar there, and the
		planar estimate gives the stronger bound. Otherwise, the zeros of \(\tau\)
		are isolated and, hence, finite on the compact support after shrinking.  Away
		from these zeros we are in the nowhere vanishing torsion case.  It remains to show for a small neighborhood of one zero, say \(t=0\).
		
		Set
		\[
		P_0=\operatorname{span}\{\gamma'(0),\gamma''(0)\}.
		\]
		After a rigid motion, translation, and analytic reparametrization, absorbing
		the Jacobian into the amplitude, we may write
		\[
		\gamma(s)=(x(s),y(s),z(s))
		=
		\bigl(s+O(s^2),\,c_2s^2+O(s^3),\,c_ks^k+O(s^{k+1})\bigr),
		\]
		where
		\[
		c_2c_k\neq0,\qquad k\ge4.
		\]
		If no finite \(k\) occurs, then \(z\equiv0\) near \(0\), so the piece is
		planar and is already handled.
		
		Let
		\[
		K_0(h)=\int \chi_0(s)e^{2\pi i h\cdot\gamma(s)}\,ds
		\]
		with \(\chi_0\) supported in this neighborhood.  Write
		\[
		\rho=1+|h|,
		\qquad
		V=1+|\pi_{P_0}h|.
		\]
		We claim
		\begin{equation}\label{1}
			|K_0(h)|
			\lesssim
			\begin{cases}
				\rho^{-1/k}, & V\le \rho^{2/k},\\[2mm]
				\rho^{-\frac1{3(k-2)}}V^{-\frac{k-3}{3(k-2)}},
				& V>\rho^{2/k}.
			\end{cases}
		\end{equation}
		Indeed,
		\[
		\det(\gamma'(0),\gamma''(0),\gamma^{(k)}(0))\neq0,
		\]
		so, after shrinking the support,
		\[
		\max_{1\le j\le k}|h\cdot\gamma^{(j)}(s)|\gtrsim |h|.
		\]
		The finite-type van der Corput gives the first bound in \eqref{1}.
		
		Assume now \(V>\rho^{2/k}\).  The case \(V=O(1)\) is trivial, so assume
		\(V\approx |\pi_{P_0}h|\).  Write \(h=(q,m)\), where
		\(q\in P_0\) and \(m\in P_0^\perp\), and set \(p(s)=(x(s),y(s))\). The phase is 
		\[h\cdot \gamma(s)=q\cdot p(s)+mz(s).\] We split the integral at the scale
		\[
		\delta=\left(\frac{V}{\rho}\right)^{1/(k-2)}.
		\]
		We discuss in two cases for \(|s|\le c\delta\) and \(|s|\approx r\), \(r\ge c\delta\).
		
		\noindent \textbf{Case 1. } On \(|s|\le c\delta\), the projected curve \(p\) has nowhere vanishing curvature,
		so
		\[
		\max\{|q\cdot p'(s)|,\ |q\cdot p''(s)|\}\gtrsim |q|\approx V.
		\]
		Also,
		\[
		|mz'(s)|\lesssim \rho |s|^{k-1},
		\qquad
		|mz''(s)|\lesssim \rho |s|^{k-2}.
		\]
		Since $|s|\le c\delta$ and $\rho \delta^{k-2}=V$, we have
		\[\max\{|mz'(s)|,	|mz''(s)|\}\ls c^{k-2}V.\]
		Choosing \(c>0\) small gives
		\[
		\max\{|(h\cdot\gamma)'(s)|,\ |(h\cdot\gamma)''(s)|\}\gtrsim V
		\]
		on the inner region.  Hence the inner contribution is
		\[
		\lesssim V^{-1/2}.
		\]
		
		\noindent \textbf{Case 2.}	On a dyadic annulus \(|s|\approx r\), \(r\ge c\delta\), the normal form gives
		\[
		|\det(\gamma'(s),\gamma''(s),\gamma'''(s))|\approx r^{k-3}.
		\]
		Since the first three derivatives of \(\gamma\) are bounded, this implies
		\[
		\max_{1\le j\le3}|h\cdot\gamma^{(j)}(s)|
		\gtrsim \rho r^{k-3}.
		\]
		Thus the annular contribution is
		\[
		\lesssim (\rho r^{k-3})^{-1/3}.
		\]
		Summing over dyadic \(r\ge c\delta\) gives
		\[
		\sum_{r\ge c\delta}(\rho r^{k-3})^{-1/3}
		\lesssim
		(\rho\delta^{k-3})^{-1/3}.
		\]
		Moreover \(V^{-1/2}\le(\rho\delta^{k-3})^{-1/3}\) when
		\(V>\rho^{2/k}\).  Therefore
		\[
		|K_0(h)|
		\lesssim
		(\rho\delta^{k-3})^{-1/3}
		=
		\rho^{-\frac1{3(k-2)}}V^{-\frac{k-3}{3(k-2)}}.
		\]
		This proves the claim \eqref{1}.
		
		We now prove the Schur bound for \(K_0\).  Fix
		\(\xi\in\mathcal E_\lambda\), put \(h=\xi-\eta\), and decompose dyadically
		according to
		\[
		\rho\le 1+|h|<2\rho,
		\qquad
		1\le \rho\lesssim \lambda.
		\]
		For the subregion \(V\le \rho^{2/k}\), using \eqref{1} and the projection
		count,
		\[
		\begin{aligned}
			S_\rho^{(0)}
			&\lesssim
			\rho^{-1/k}N_\xi(C\rho^{2/k};P_0)                         \\
			&\lesssim_\varepsilon
			\lambda^\varepsilon
			\rho^{-1/k}\bigl(1+(\lambda\rho^{2/k})^{1/2}\bigr)          \\
			&\lesssim_\varepsilon \lambda^{1/2+\varepsilon}.
		\end{aligned}
		\]
		For the subregion \(V>\rho^{2/k}\), set
		\[
		\alpha=\frac1{3(k-2)},
		\qquad
		\beta=\frac{k-3}{3(k-2)}.
		\]
		Then \(\alpha+\beta=1/3\) and \(0<\beta<1/2\).  Dyadic summation in
		\(U\approx V\) gives
		\[
		\begin{aligned}
			S_\rho^{(1)}
			&\lesssim
			\sum_{\substack{U>\rho^{2/k}\\ U\ \mathrm{dyadic}\\ U\lesssim\rho}}
			\rho^{-\alpha}U^{-\beta}N_\xi(2U;P_0)                     \\
			&\lesssim_\varepsilon
			\lambda^\varepsilon\rho^{-\alpha}
			\sum_{\substack{U>\rho^{2/k}\\ U\ \mathrm{dyadic}\\ U\lesssim\rho}}
			U^{-\beta}\bigl(1+(\lambda U)^{1/2}\bigr)                  \\
			&\lesssim_\varepsilon
			\lambda^\varepsilon
			+
			\lambda^{1/2+\varepsilon}\rho^{-\alpha}\rho^{1/2-\beta}   \\
			&
			\lesssim_\varepsilon
			\lambda^{2/3+\varepsilon}.
		\end{aligned}
		\]
		Thus, for each dyadic \(\rho\),
		\[
		S_\rho^{(0)}+S_\rho^{(1)}
		\lesssim_\varepsilon \lambda^{2/3+\varepsilon}.
		\]
		After summing over \(O(\log\lambda)\) dyadic \(\rho\)'s and absorbing the
		logarithm into \(\lambda^\varepsilon\), we obtain
		\[
		\sup_{\xi\in\mathcal E_\lambda}
		\sum_{\eta\in\mathcal E_\lambda}|K_0(\xi-\eta)|
		\lesssim_\varepsilon
		\lambda^{2/3+\varepsilon}.
		\]
		Combining the finitely many pieces gives
		\[
		B_\lambda\lesssim_\varepsilon \lambda^{2/3+\varepsilon}.
		\]
		Schur's test therefore yields
		\[
		\|e_\lambda\|_{L^2(\Sigma)}
		\lesssim_\varepsilon
		B_\lambda^{1/2}\|e_\lambda\|_{L^2(\mathbb T^3)}
		\lesssim_\varepsilon
		\lambda^{1/3+\varepsilon}
		\|e_\lambda\|_{L^2(\mathbb T^3)}.
		\]
	\end{proof}
	
	\subsection{Discrete Restriction Conjecture}
	We briefly discuss the relation between the $L^2$ restriction bounds and global $L^p$ bounds for eigenfunctions, such as the Discrete Restriction Conjecture. This provides a different perspective for understanding the bounds in \eqref{T3} and \eqref{T4}. See Remark \ref{dcrem}.
	
	\begin{proposition}
		Let \(d\ge 2\), and let \(\Sigma\subset \T^d\) be a fixed smooth compact
		submanifold of codimension $m$.
		Assume that for some \(2\le p\le \infty\) one has the global estimate
		\begin{equation}\label{Lp-assumption}
			\|e_\lambda\|_{L^p(\T^d)}
			\le A_p(\lambda)
			\|e_\lambda\|_{L^2(\T^d)} .
		\end{equation}
		Then
		\begin{equation}\label{general-restriction-bound}
			\|e_\lambda\|_{L^2(\Sigma)}
			\lesssim
			\lambda^{\frac{m}p}
			A_p(\lambda)
			\|e_\lambda\|_{L^2(\T^d)}.
		\end{equation}
	\end{proposition}
	\begin{proof}
		Let
		\[
		T_{\lambda^{-1}}(\Sigma)
		=
		\{x\in\T^d:\dist(x,\Sigma)\le C\lambda^{-1}\}
		\]
		be a fixed constant multiple of the wavelength-scale tube around \(\Sigma\).
		The first step is the local tube estimate
		\begin{equation}\label{tube-submean-d}
			\int_\Sigma |e_\lambda|^2\,d\sigma
			\lesssim
			\lambda^{m}
			\int_{T_{\lambda^{-1}}(\Sigma)}
			|e_\lambda(x)|^2\,dx .
		\end{equation}
		
		We prove \eqref{tube-submean-d}.  Fix \(x\in\T^d\), and rescale around \(x\)
		at wavelength scale:
		\[
		u(z)=e_\lambda\left(x+\frac{z}{\lambda}\right).
		\]
		Then \(u\) satisfies
		\[
		-\Delta_z u=(2\pi)^2u
		\]
		on a fixed Euclidean ball.  By the standard interior elliptic estimate,
		\[
		|u(0)|^2
		\lesssim
		\int_{B(0,1)}|u(z)|^2\,dz.
		\]
		Returning to \(e_\lambda\), this gives
		\begin{equation}\label{pointwise-submean-d}
			|e_\lambda(x)|^2
			\lesssim
			\lambda^d
			\int_{B(x,\lambda^{-1})}
			|e_\lambda(y)|^2\,dy .
		\end{equation}
		The factor \(\lambda^d\) comes from the change of variables
		\(y=x+z/\lambda\).
		
		Now integrate \eqref{pointwise-submean-d} over \(x\in\Sigma\):
		\[
		\begin{aligned}
			\int_\Sigma |e_\lambda(x)|^2\,d\sigma(x)
			&\lesssim
			\lambda^d
			\int_\Sigma
			\int_{B(x,\lambda^{-1})}
			|e_\lambda(y)|^2\,dy\,d\sigma(x)     \\
			&=
			\lambda^d
			\int_{\T^d}
			|e_\lambda(y)|^2
			\sigma\{x\in\Sigma:\dist(x,y)\le \lambda^{-1}\}
			\,dy .
		\end{aligned}
		\]
		Let $k=d-m$. Since \(\Sigma\) is a fixed smooth \(k\)-dimensional submanifold, a ball of
		radius \(\lambda^{-1}\) intersects \(\Sigma\) in \(k\)-dimensional measure
		\[
		\lesssim \lambda^{-k}.
		\]
		Moreover, the inner set is empty unless
		\(y\in T_{\lambda^{-1}}(\Sigma)\).  Therefore
		\[
		\int_\Sigma |e_\lambda|^2\,d\sigma
		\lesssim
		\lambda^d\lambda^{-k}
		\int_{T_{\lambda^{-1}}(\Sigma)}
		|e_\lambda(y)|^2\,dy,
		\]
		which proves \eqref{tube-submean-d}.
		
		Next, since \(\Sigma\) has codimension $m$ inside \(\T^d\), its
		\(\lambda^{-1}\)-tube has volume
		\begin{equation}\label{tube-volume-d}
			|T_{\lambda^{-1}}(\Sigma)|
			\lesssim
			\lambda^{-m} .
		\end{equation}
		By H\"older's inequality and \eqref{Lp-assumption},
		\[
		\begin{aligned}
			\int_{T_{\lambda^{-1}}(\Sigma)}
			|e_\lambda(x)|^2\,dx
			&\le
			|T_{\lambda^{-1}}(\Sigma)|^{1-2/p}
			\|e_\lambda\|_{L^p(\T^d)}^2\\
			&\lesssim
			\lambda^{-m(1-2/p)}
			\|e_\lambda\|_{L^p(\T^d)}^2\\
			&\ls 	\lambda^{-m(1-2/p)}
			A_p(\lambda)^2
			\|e_\lambda\|_{L^2(\T^d)}^2 .
		\end{aligned}
		\]
		Combining this with \eqref{tube-submean-d} gives
		\[
		\begin{aligned}
			\int_\Sigma |e_\lambda|^2\,d\sigma
			&\lesssim
			\lambda^{m}
			\lambda^{-m(1-2/p)}
			A_p(\lambda)^2
			\|e_\lambda\|_{L^2(\T^d)}^2        \\
			&=
			\lambda^{2m/p}
			A_p(\lambda)^2
			\|e_\lambda\|_{L^2(\T^d)}^2 .
		\end{aligned}
		\]
		This yields
		\eqref{general-restriction-bound}.
	\end{proof}
	
	\begin{corollary}
		Assume \(d\ge3\), and suppose that the endpoint discrete restriction
		estimate
		\begin{equation}\label{endpoint-discrete}
			\|e_\lambda\|_{L^{p_c}(\T^d)}
			\lesssim_\eps
			\lambda^\eps
			\|e_\lambda\|_{L^2(\T^d)},
			\qquad
			p_c=\frac{2d}{d-2},
		\end{equation}
		holds.  Then for every smooth compact submanifold
		\(\Sigma\subset\T^d\) of codimension $m$,
		\begin{equation}\label{endpoint-corollary}
			\|e_\lambda\|_{L^2(\Sigma)}
			\lesssim_\eps
			\lambda^{\frac m2-\frac md+\eps}
			\|e_\lambda\|_{L^2(\T^d)} .
		\end{equation}
	\end{corollary}

	\begin{remark}\label{dcrem}
		For every smooth curve \(\Sigma\subset\T^d\), we have $m=d-1$ and then the endpoint  estimate \eqref{endpoint-discrete}
		would imply
		\begin{equation}\label{dcb}
			\|e_\lambda\|_{L^2(\Sigma)}
			\lesssim_{\eps}
			\lambda^{\frac{d-3}2+\frac1d+\eps}
			\|e_\lambda\|_{L^2(\T^d)}.
		\end{equation}
		The power is exactly $\frac13$ when $d=3$ and $\frac34$ when $d=4$, which agrees with the power that we have obtained in \eqref{T3} and \eqref{T4}. Since the conjecture \eqref{endpoint-discrete} is still open (see Bourgain-Demeter \cite[Conjecture 2.6]{BD} and references therein), it would be interesting to prove \eqref{T3} for any smooth curve \(\Sigma\subset\T^3\), without assuming \eqref{endpoint-discrete}.
	\end{remark}

	\bibliographystyle{plain}
	
\end{document}